\documentclass[a4paper]{amsart}
\usepackage[latin9]{inputenc}
\usepackage{amsfonts}
\usepackage{amssymb}
\usepackage{amsmath}
\usepackage{amsthm}
\usepackage{color}
\usepackage{geometry}
\usepackage{verbatim}
\usepackage[toc,page]{appendix}
\begin{document}
\providecommand{\keywords}[1]{\textbf{\textit{Keywords: }} #1}
\newtheorem{theorem}{Theorem}[section]
\newtheorem{lemma}[theorem]{Lemma}
\newtheorem{prop}[theorem]{Proposition}
\newtheorem{kor}[theorem]{Corollary}
\theoremstyle{definition}
\newtheorem{defi}[theorem]{Definition}
\theoremstyle{remark}
\newtheorem{remark}[theorem]{Remark}
\newtheorem{problem}[theorem]{Problem}
\newtheorem{question}[theorem]{Question}
\newtheorem{conjecture}[theorem]{Conjecture}
\newtheorem{example}[theorem]{Example}
\newtheorem{condenum}[theorem]{Condition}

\newcommand{\cc}{{\mathbb{C}}}   
\newcommand{\ff}{{\mathbb{F}}}  
\newcommand{\nn}{{\mathbb{N}}}   
\newcommand{\qq}{{\mathbb{Q}}}  
\newcommand{\rr}{{\mathbb{R}}}   
\newcommand{\zz}{{\mathbb{Z}}}  
\newcommand{\fp}{{\mathfrak{p}}}

\newcommand\JK[1]{{\color{blue}{#1}}}

\title[Rank gain of Jacobians]{Rank gain of Jacobians over number field extensions with prescribed Galois groups}

\date{\today}

\thanks{Bo-Hae Im was supported by the National Research Foundation of Korea(NRF) grant
funded by the Korea government(MSIT) (No.~2020R1A2B5B01001835).
}

\author{Bo-Hae Im}
\address{
Department of Mathematical Sciences, KAIST, 291 Daehak-ro, Yuseong-gu, Daejeon, 34141, South Korea, 
and\\
Korea Institute for Advanced Study, Hoegiro 87, Seoul, 130-722, South Korea}

\email{bhim@kaist.ac.kr}
\author{Joachim K\"onig}
\address{Department of Mathematics Education, Korea National University of Education, Cheongju, 28173, South Korea}
\email{jkoenig@knue.ac.kr}
\begin{abstract}
We investigate the rank gain of elliptic curves, and more generally, Jacobian varieties, over non-Galois extensions whose Galois closure has Galois group permutation-isomorphic to a prescribed group $G$ (in short, ``$G$-extensions"). 
In particular, for alternating groups and (an infinite family of) projective linear groups $G$, we show that most elliptic curves over (e.g.) $\mathbb{Q}$ gain rank over infinitely many $G$-extensions, conditional only on the parity conjecture.
More generally, we provide a theoretical criterion which allows to deduce that ``many" elliptic curves gain rank over infinitely many $G$-extensions, conditional on the parity conjecture and on the existence of geometric Galois realizations with group $G$ and certain local properties. 
\end{abstract}
\maketitle

\section{Introduction}

The rank of elliptic curves over number fields has been one of the main subjects in the areas of algebraic number theory and arithmetic geometry. In particular, questions about the rank of quadratic twists of a given elliptic curve $E$ over a number field $K$ have been studied extensively, and are directly related to the question over which quadratic extensions $K(\sqrt{a})/K$ the rank of $E$ grows, upon recalling that for an elliptic curve $E/K$ over a number field $K$ and its quadratic twist $E^a$ by a square-free integer $a$, the rank of $E$ over the quadratic extension $K(\sqrt{a})$ is equal to the sum of the ranks of $E$ and $E^a$ over $K$, which implies that the positive rank of the twist over $K$ gives the rank growth of $E$ over the quadratic extension of $K$. For example, in \cite{RS}, a lower bound of density of twists of elliptic curves over $\mathbb{Q}$ with positive rank has been given, and in \cite{CJ, ILo, I}, the simultaneous rank growth of more than one elliptic curve over given quadratic extensions has been studied.  Some conditional results on the positive rank of twists of elliptic curves, relative to the Birch and Swinnerton-Dyer conjecture and the finiteness of the Tate-Shafarevich group, have been given in \cite{MR2, DD2}.

In this paper, we more generally investigate the rank growth of elliptic curves over extensions of the base field whose Galois closure has a prescribed Galois group. 

Given a transitive permutation group $G\le S_n$, by a $G$-extension $F/K$ we mean an (in general, non-Galois) degree-$n$ extension over $K$ whose Galois closure has Galois group permutation-isomorphic to $G$.

The following are variants (of increasing strength) of the same question of interest:
\begin{question}
\label{ques:strong}
\label{ques:weak}
Let $K$ be a number field and $G\le S_n$ a transitive permutation group.
\begin{itemize}
    \item[a)] Do there exist infinitely many elliptic curves $E_i$ ($i\in I$) over $K$, pairwise non-isomorphic over $\mathbb{C}$, such that each $E_i$ gains rank over infinitely many pairwise linearly disjoint $G$-extensions of $K$?
    \item[b)] Can the infinite family $(E_i)_{i\in I}$ in a) even be assumed to contain ``most" elliptic curves over $K$ (with a suitable way of counting)?
    \item[c)] Can it even be assumed to contain all elliptic curves over $K$?
\end{itemize}
\end{question}

In the key case $K=\mathbb{Q}$, the strongest version Question \ref{ques:strong}c) is known to have a positive answer for all symmetric groups $G=S_n$ (e.g., see \cite{LOT}).

Examples where Question \ref{ques:weak}a) is known to have a positive answer include $G=C_3$, (for example, see~\cite{IW18, DFK04}) or $G=AGL_1(p)$ (the normalizer of an order $p$ subgroup in $S_p$) for $p\ge 3$ (see \cite[Corollary 7]{Dokch};  our translation of the result in \cite{Dokch} merely uses that $AGL_1(p)$ is the Galois group of, e.g., $X^p-q$ for prime numbers $p$ and $q$).
On the other hand for $G=C_p$ with a prime $p\ge 7$, Question \ref{ques:weak}a) is conjectured to have a negative answer (for example, see \cite{MR20}).  

Note that if a $G$-extension $F/K$ has a proper intermediate extension $F_0/K$, then the rank gain of an elliptic curve over $F$ might sometimes occur for trivial reasons (namely, due to the rank gain already over $F_0$). We are therefore primarily interested in the case where such intermediate extensions do not exist. Via the Galois correspondence, this is the case if and only if the point stabilizer $G_1\le G$ is a maximal subgroup inside $G$, or equivalently, if the transitive action of $G$ is in fact a {\it primitive} action. In the same spirit, the condition of linear disjointness is included in Question \ref{ques:strong} 
to exclude certain trivial answers; it is, however, furthermore useful for theoretical purposes, since it guarantees that the rank of the respective elliptic curve becomes infinite over the compositum of these extensions. See, e.g., \cite{I06, ILa08, AGP}  for consideration about infinite rank over such ``large" fields.

\section{Main results}


We attack Question \ref{ques:weak} (in fact, for very broad classes of elliptic curves), using the root number $W(E/K)$ of an elliptic curve defined over $K$. See Section \ref{sec:rootnumber} for more on root numbers; for now it suffices to note that $W(E/K)$ equals $-1$ or $1$, and conditional on the Birch and Swinnerton-Dyer conjecture, equals the parity $(-1)^{rank(E(K))}$ of the rank of $E$ over $K$.

We fix a finite transitive permutation group $G\le S_n$, a number field $K$, an  elliptic curve $E$ over $K$, and ask the following question: As $F/K$ runs through all (degree-$n$) $G$-extensions of $K$, does the sign $W(E/F)$ differ from $W(E/K)$ infinitely often? If so, that means in particular that the analytic rank of $E$ over $F$ is higher than the one over $K$ for infinitely many $F$, and assuming the parity conjecture, the rank of $E$ over $F$ is larger than the one over $K$. Our goal is to obtain a positive answer for as broad a class of elliptic curves as possible, for prescribed groups $G$. %

We focus in particular on non-solvable groups. One of the most successful approaches to realizing such groups as Galois groups over number fields $K$ is via $K$-regular realizations over the function field $K(t)$ (and subsequent application of Hilbert's irreducibility theorem), via, e.g., the so-called rigidity method (here, a field extension $L/K(t)$ is called $K$-regular, if $K$ is algebraically closed in $L$; cf.\ Section \ref{sec:basic_ext}).
We provide the following general criterion, which enables us to give positive answers to Question~\ref{ques:weak} for groups $G$ conditional on $K$-regular realizations over $K(t)$, possibly with certain local conditions.

\begin{theorem}
\label{thm:grunwald}
Let $K$ be a number field, and let $G\le S_n$ be a transitive permutation group occurring as the Galois group of a $K$-regular Galois extension $L/K(t)$. 
 Assume that one of the following holds:
\begin{itemize}
\item[i)] $G$ is not contained in $A_n$, or
\item[ii)] there exists a $K$-rational 
 branch point $t_i\in \mathbb{P}^1(K)$ of $L/K(t)$ with the following property: if $I\trianglelefteq D$ denote the inertia and decomposition group at $t\mapsto t_i$ in $L/K(t)$, there exists $x\in D$ such that $n-|orb(\langle I,x\rangle)|$ is odd, where $orb(H)$ is the set of orbits of a subgroup $H\le S_n$.
\end{itemize}
Then under the parity conjecture, there exists a positive density set $\mathcal{S}$ of primes of $K$ such that, for each $p\in \mathcal{S}$, each elliptic curve over $K$ with multiplicative reduction at $p$ gains rank over infinitely many pairwise linearly disjoint $G$-extensions of $K$.
\end{theorem}

Theorem \ref{thm:grunwald} means in particular that groups not contained in $A_n$ are typically easier to deal with; see also Theorem \ref{thm:generic}. Our remaining applications are thus for subgroups of $A_n$.
Regarding the somewhat technical Case ii), we emphasize that it may often be verified purely theoretically, without relying on an explicit polynomial for the Galois extension $L/K(t)$. See Section \ref{sec:firstex} for a first demonstration.
As a broader application of Theorem \ref{thm:grunwald}, we show the following result, providing rank gain over extensions of degree $p+1$ with Galois group (of the Galois closure) equal to $PSL_2(p)$ in its natural action.

\begin{theorem}
\label{thm:psl}
\label{thm:psl_strong}
Let $K$ be a number field not containing $\sqrt{-1}$, let $p$ be a prime number congruent to $3$ mod $8$, and let 
 $G=PSL_2(p)\le S_{p+1}$.
Then under the parity conjecture, for all but finitely many primes $q$ of $K$ inert in $K(\sqrt{-1})/K$ (with the exceptional set depending on $p$), the following holds:\\
Every elliptic curve over $K$ with multiplicative reduction at $q$ gains rank over infinitely many (pairwise linearly disjoint) degree-$p+1$ $G$-extensions of $K$.
\end{theorem}

For the case $G=A_n$, we reach an even stronger conclusion.
\begin{theorem}
\label{thm:an_cond}
Let $K$ be a number field, and let $E$ be an elliptic curve over $K$ having at least one prime of multiplicative reduction. 
Then under the parity conjecture, for each $n\ge 4$, there exist infinitely many pairwise linearly disjoint degree-$n$ extensions $F/K$ with Galois group $A_n$ such that $E$ gains rank over $F$.

\end{theorem}


The proofs of the above results are contained in Section \ref{sec:pf_main}. The general idea of proof is to decompose the problem into two parts: a purely group-theoretical condition, and a problem about realization of the prescribed group $G$ as a Galois group with certain local conditions. One way to guarantee the latter is to use previous results about the local behavior of specializations of Galois extensions of function fields such as \cite[Theorem 4.1]{KLN19}, which is a main ingredient for the proof of Theorem \ref{thm:grunwald}. 

In fact, the proofs yield an asymptotic result. In all cases above, there exists a constant $a>0$ (depending on $G$ and $K$, but {\it not} on the concrete elliptic curve $E$!) satisfying the following:  
the number of $G$-extensions $F/K$ with relative discriminant of norm $\le B$ such that $E$ gains rank over $F$ is asymptotically at least $B^a$, cf.\ Remark \ref{rem:density}.

For some small alternating and linear groups, we also give explicit sample polynomials over whose root fields the rank of a prescribed elliptic curve will grow. In Section~\ref{sec:density}, we briefly discuss the frequency of the elliptic curves satisfying the conditions of the above theorems; it turns out that most elliptic curves over $\mathbb{Q}$, when ordered by height, satisfy the assumptions, meaning that the results may be interpreted as answers to Question~\ref{ques:strong}b). It should be noted that there are known methods to obtain {\it unconditional} rank gain results for elliptic curves $E$ in certain scenarios; notably, via the construction of covers $E\to \mathbb{P}^1$ with Galois group $G$. These have been used for special cases in, e.g., \cite{IW18} and \cite{LOT}. However, this approach is only available for rather restricted classes of groups (and then in general suffices only to answer the ``weak" Question~\ref{ques:weak}a)). Finally, note that in order to attack the ``strongest" Question~\ref{ques:strong}c) in full generality, root number considerations alone cannot be expected to suffice either; e.g., already for the case $G=C_2$, there are known examples of elliptic curves over number fields with the same root number over all quadratic extensions \cite{DD2}. 

In Section \ref{sec:abvar}, we use work of Brumer et al. \cite{BKS} to give a generalization from elliptic curves to Jacobians of curves of arbitrary genus (although with certain technical assumptions). This allows us to generalize results such as Theorem \ref{thm:an_cond} to, e.g., Jacobians of hyperelliptic curves. Concretely, we show:

\begin{theorem}
\label{thm:hyper}
Assume the parity conjecture for abelian varieties. 
Let $f(X)\in \zz[X]$ be a separable polynomial of degree $d\ge 5$, and $C$ be the hyperelliptic curve $Y^2=f(X)$. Assume that there exists an odd prime $p$ such that $f$ mod $p$ is of degree $d$ and factors as $(X-a)^2g(X)$ where $g$ is separable and coprime to $X-a$.\footnote{This condition is automatically fulfilled if $p$ is a simple root of the discriminant $\Delta(f)$.} Then 
for every integer $n\ge 4$, there exist infinitely many (degree-$n$) $A_n$-extensions of $\mathbb{Q}$ over which the Jacobian $J(C)$ gains rank.
\end{theorem}

See also Theorem \ref{thm:semist_jac} for a more general criterion.

\section{Prerequisites}

\subsection{Completion and specialization of $G$-extensions}
\label{sec:basic_ext}
Let $F$ be a field and $L/F$ a separable (but not necessarily Galois) extension of degree $n$ with Galois closure $\Omega/F$. We call $L/F$ a $G$-extension if $Gal(\Omega/F)$ is a permutation group isomorphic to $G$ via the usual action on conjugates of $L$. 

Now let $R$ be a Dedekind domain with field of fractions $F$. Let $p$ be a prime ideal of $R$ with perfect residue field, which we denote by $F_p$.
Denote by $L_p$ the direct product of residue fields $L_{\mathfrak{p}}$ where $\mathfrak{p}$ runs through the prime ideals extending $p$ in the integral closure of $R$ in $L$. This is an \'etale algebra (in the case where $L/F$ is Galois, it is a direct product of isomorphic Galois extensions, and for most purposes, we may simply work with one of those extensions and refer to this as $L_p$). In the same way, the completion $\widehat{L_p}$ at $p$ may be defined as the direct product of completions $\widehat{L_{\mathfrak{p}}}$ at places $\mathfrak{p}$ extending $p$. By the decomposition group of (the Galois closure) $L/F$ at $p$, we mean the decomposition group at any prime extending $p$, which yields a subgroup of $Gal(\Omega/F)$ well-defined up to conjugacy. Notably, the degrees of the completions $\widehat{L_\mathfrak{p}}/\widehat{F_p}$ equal the orbit lengths of the decomposition group at $p$ in the Galois closure $\Omega/F$ of $L/F$.

We are particularly interested in the case where $F=k(t)$ is a rational function field.
A $G$-extension $L/k(t)$ is called {\it $k$-regular}, 
if the algebraic closure of $k$ inside $L$ equals $k$.

Now let $k$ be of characteristic zero, and let $L/k(t)$ be a $k$-regular $G$-extension with Galois closure $\Omega/k(t)$. 
For each $t_i\in \overline{k}\cup\{\infty\}$, the {\it ramification index} of $\Omega/k(t)$ at $t_i$ is the minimal positive integer $e_i$ such that $\Omega$ embeds into $\overline{k}(((t-t_i)^{1/e_i}))$. Here, if $t_i=\infty$, one should replace $t-t_i$ by $1/t$.  If the ramification index is larger than $1$, then $t_i$ is called a {\it branch point} of $L/k(t)$.
The set of branch points is always a finite set.
To each branch point $t_i$ (of ramification index $e_i$), we may associate a well-defined conjugacy class $C_i$ of $G$, corresponding to the automorphism $(t-t_i)^{1/e_i}\mapsto \zeta (t-t_i)^{1/e_i}$ of 
$\overline{k}(((t-t_i)^{1/e_i}))$, where $\zeta$ is a primitive $e_i$-th root of unity. The ramification index then equals the order of any element in the class $C_i$. 

For any $t_0\in k\cup\{\infty\}$, 
the {\it specialization} of $L/k(t)$ at $t_0$, denoted by $L_{t_0}/k$ is the residue extension of $L/k(t)$ at $t\mapsto t_0$ (where the implicit Dedekind domain $R\subset k(t)$ should be taken as the local ring corresponding to the place $t\mapsto t_0$).

Now let $K$ be a number field. For our purposes, it is important to gain information about the local behaviour of specializations $F_{t_0}/K$ at primes of $K$. The following two results, ensuring realizability of certain unramified, resp., ramified local behaviours, are crucial for us. The first one is a special case of \cite[Theorem 1.2]{DG12}, the second one is  \cite[Theorem 4.1 and Remark 4.2(1)]{KLN19}.

\begin{theorem}
\label{thm:dg}
Let $L/K(t)$ be a $K$-regular Galois extension with Galois group $G$. There exists a finite set $\mathcal{S}_0 = \mathcal{S}_0(L/K(t))$ of primes of $K$, depending only on $L/K(t)$, such that the following holds. If $p$ is any prime of $K$ not contained in $\mathcal{S}_0$ and $C$ is any conjugacy class of $G$, then there exist infinitely many $t_0\in K$ such that the specialization $L_{t_0}/K$ is unramified at $p$ with Frobenius class $C$.
\end{theorem}

\begin{theorem}
\label{thm:spec_fctfd}
Let $K$ be a number field, $G$ be a finite group and $L/K(t)$ be a $K$-regular Galois extension with Galois group $G$. There exists a finite set $\mathcal{S}_0 = \mathcal{S}_0(L/K(t))$ of primes of $K$, depending only on $L/K(t)$, such that the following holds.

Let $t_i\in \mathbb{P}^1(\overline{K})$ be a branch point of $L/K(t)$, and denote by $I$ and $D$ the inertia group and decomposition group of $L/K(t)$ at (the prime ideal corresponding to) $t\mapsto t_i$. Let $p$ be a prime of $K$, not contained in $\mathcal{S}_0$, and such that there exists a prime $q$ extending $p$ of relative degree $1$ in $K(t_i)$. Choose $x\in D$ such that the image of $x$ under the canonical projection $D\to D/I$ equals the Frobenius of $L(t_i)_{t_i}/K(t_i)$ at $q$ (where $D/I$ is identified with the Galois group of the specialization $L(t_i)_{t_i}/K(t_i)$). 
Then there exist infinitely many $t_0\in K$ such that the specialization $L_{t_0}/K$ has inertia group conjugate to $I$ and decomposition group conjugate to $\langle I,x\rangle$ at $p$.
\end{theorem}

\begin{remark}
\label{rem:explicit_primes}
The exceptional sets $\mathcal{S}_0$ in Theorems \ref{thm:dg} and \ref{thm:spec_fctfd} are not identical, although both can be made effective (see, e.g., \cite[Theorem 2.2]{KN20}). This becomes relevant (via the proof of Theorem \ref{thm:grunwald}) when applying our main results to concrete elliptic curves; see, e.g., Example \ref{ex:psl}. We refrain from stating an effective version here in order to avoid overly technical notation.
\end{remark}

Finally, we will make use of the following result, which is a consequence of the well-known compatibility of Hilbert's irreducibility theorem with the weak approximation property (cf., e.g., \cite[Proposition 2.1]{PV}.
\begin{prop}
\label{prop:hilbert_wa}
Let $L/K(t)$ be a $K$-regular Galois extension with Galois group $G$, and let $\mathcal{S}$ be a finite set of primes of $K$. For each $p\in \mathcal{S}$ choose $t(p)\in K$, not in the set of branch points of $L/K(t)$, and let $\widehat{(L_{t(p)})_p}/\widehat{K_p}$ be the completion of the specialization $L_{t(p)}/K$ at $p$. Then there exist infinitely many values $t_0\in K$ such that
\begin{itemize}
    \item[i)] $Gal(L_{t_0}/K) = G$, and
    \item[ii)] for each $p\in \mathcal{S}$, $\widehat{(L_{t_0})_p}/\widehat{K_p} \cong \widehat{(L_{t(p)})_p}/\widehat{K_p}$.
\end{itemize}
Moreover, the extensions $L_{t_0}/K$ thus obtained may be assumed pairwise linearly disjoint.
\end{prop}
\begin{remark}
\label{rem:density}
Concretely, Proposition \ref{prop:hilbert_wa} may be shown by using Krasner's lemma, showing that in the intersection of sufficiently small open $p$-adic neighborhoods ($p\in \mathcal{S}$), the local behavior of the specializations is unchanged at all $p\in \mathcal{S}$, which is assertion ii). Since this intersection is a non-thin set (in the sense of Serre), it contains infinitely many values $t_0$ preserving the Galois group $G$.
The above also allows an asymptotic estimate of the number of extensions $L_{t_0}/K$ fulfilling the conclusion of Proposition \ref{prop:hilbert_wa}: For $B\in \mathbb{N}$, denote by $N(L/K(t), B, \{(p,t(p))\mid p\in \mathcal{S}\})$ be the number of distinct field extensions $L_{t_0}/K$ fulfilling the conclusion of Proposition \ref{prop:hilbert_wa} for the given data, and with relative discriminant $\Delta(L_{t_0}/K)$ of norm at most $B$. Then there exists a constant $a:=a(L,K)>0$, depending on the genus of $L$ (and thus, on $G$) and on $K$, but not on $\mathcal{S}$, such that for all sufficiently large $B$, one has 
$$N(L/K(t), B, \{(p,t(p))\mid p\in \mathcal{S}\}) \ge B^a.$$
Indeed, this is a direct consequence of \cite[Corollary 3.3]{BG}.
\end{remark}

\subsection{Root numbers of elliptic curves and abelian varieties}
\label{sec:rootnumber}
We recall some basics about local root numbers of abelian varieties, and in particular of elliptic curves.
We refer to \cite{BKS} for the following.

Let $A$ be an abelian variety defined over a number field $K$. The (global) \textit{root number} $W(A):=W(A/K) \in \{-1,1\}$ is the sign in the (conjectural, except for special cases) functional equation for the $L$-function of $A$. By definition $W(A)$ is the product of {\it local root numbers} $W_p(A)$, i.e., the root numbers associated to $A\otimes_K K_p$, where $p$ runs through all primes of $K$, including the archimedean ones (note that $W_p(A)=1$ for all but finitely many $p$). The parity conjecture for abelian varieties states that $W(A) = (-1)^{rank(A(K))}$, where $rank(A(K))$ is the rank of the Mordell-Weil group of $A$ over $K$. Since on the other hand, trivially $rank(A(L))\ge rank(A(K))$ for any finite extension $L\supseteq K$, the root number may give information about the rank gain of $A$ upon base change to certain extensions (once again, conditional on the parity conjecture).

The following facts about local root numbers of elliptic curves are well-known, cf., e.g., \cite{Roh} (which also includes formulas for the local root numbers in the case of additive reduction).
\begin{prop}
\label{prop:locroot}
Let $E/K$ be an elliptic curve over a number field $K$ and $p$ a place of $K$. Then
$W_p(E/K) = \begin{cases}
- 1, \text{ if } p \text{ is an archimedean place },\\
+ 1, \text{ if } E/K \text{ has good reduction at } p, \\
- 1, \text{ if } E/K \text{ has split multiplicative reduction at } p\\
+ 1, \text{ if } E/K \text{ has non-split multiplicative reduction at } p\end{cases}$.
\end{prop}

In particular, in the case where $E$ is a semistable elliptic curve over $K$, one has the following characterization of the 
root number of $E$ over $K$, cf.\ e.g., \cite{Dokch}:
\begin{kor}
\label{prop:dok}
For a semistable elliptic curve over a number field $K$,
$W(E/K) = (-1)^{u+s}$, where $u$ is the number of archimedean places of $K$ and $s$ is the number of primes of $K$ at which $E$ has split multiplicative reduction.
\end{kor} 

\section{Proof of the main results}
\label{sec:pf_main}
\subsection{Some preparations}

Fix a transitive group $G\le S_n$, a $G$-extension $F/K$ of number fields, an elliptic curve $E$ defined over $K$ and a prime $p$ of $K$ at which $E$ has multiplicative reduction.
Due to Proposition \ref{prop:locroot}, the product of local root numbers $W_q(E/F)$ at primes $q$ of $F$ extending $p$ is exactly $(-1)^{s_p}$, where $s_p$ is the number of primes of $F$ extending $p$ at which $E$ has split multiplicative reduction. It is well-known (e.g., \cite[Chapter VII.5]{Silverman}) that $E$ has multiplicative reduction at all primes $\mathfrak{p}$ extending  $p$, and the reduction type at $\mathfrak{p}$ is non-split if and only if it is non-split at $p$ and the residue degree of $\mathfrak{p}$ over $p$ is odd.

On the other hand, the places extending $p$ in $F$ correspond one-to-one to the orbits of the decomposition group $D_p\le G$ at $p$ in the Galois closure $\Omega/K$ of $F/K$, and the residue degree of any such place equals the number of $I_p$-orbits ($I_p$ being the inertia group at $p$) into which the respective $D_p$-orbit decomposes.
We therefore have the following characterization:

\begin{lemma}
\label{lem:splitplaces}
Assume the above notation.
\begin{itemize}
\item[a)] If $E$ has split multiplicative reduction at $p$, then $s_p$ equals the number of orbits of the decomposition group $D_p$ at  $p$ in $\Omega/K$.
\item[b)] If $E$ has non-split multiplicative reduction at $p$, then $s_p$ equals the number of orbits of $D_p$ which split into an even number of orbits of the inertia group $I_p \trianglelefteq D_p$.
\end{itemize}
\end{lemma}

Similarly, if $q$ is an archimedean place of $K$ extended by exactly $u_q$ places of $F$, the contribution to $W(E/F)$ from places extending $q$ is simply $(-1)^{u_q}$, and this number is determined as follows.

\begin{lemma}
\label{lem:infplaces}
Assume the above notation. Let $q$ be an archimedean place of $K$.
\begin{itemize}
    \item[a)] If $q$ is real archimedean, then $u_q$ equals the number of orbits of $\tau \in G$, where $K\hookrightarrow \mathbb{R}$ is the real embedding corresponding to $q$, and $\tau$ is the complex conjugation in $\Omega/K$.
    \item[b)] If $q$ is non-real, then $u_q=[F:K]$.
\end{itemize}
\end{lemma}

In particular, the values $s_q$ in Lemma \ref{lem:splitplaces} as well as $u_q$ in Lemma \ref{lem:infplaces}a) are given group-theoretically in terms of certain orbit numbers.
Therefore, due to Proposition \ref{prop:locroot}, controlling the value of $W(E/F)$ reduces to a two-fold problem: 
firstly, a purely group-theoretical question about the existence of subgroups of $G$ with a certain orbit structure. Secondly, a local-global problem, namely finding $G$-extensions with the prescribed subgroups from the first step as decomposition (and inertia) subgroups at certain prescribed primes.

Indeed, assume that there is a prime $p$ of $K$ which is either real archimedean, or non-archimedean and of split multiplicative reduction for $E$. Assume furthermore that there exist two subgroups $H_1$ and $H_2$ of $G$ whose numbers of orbits $|orb(H_1)|$ and $|orb(H_2)|$ have different parity, and such that $H_1$ and $H_2$ both occur as decomposition groups of a suitable extension of $K_p$. 
Then $W(E/F)$ can in principle be controlled via the prime $p$; concretely, assume for the moment that $H_1$ is cyclic and contained in $A_n$ (i.e., in particular $n-|orb(H_1)|$ is even), whereas $n-|orb(H_2)|$ is odd. If there exist two $G$-extensions $F_1/K$ and $F_2/K$ with decomposition groups $H_1$ and $H_2$ at $p$, then $W_p(E/F_1)\ne W_p(E/F_2)$ due to Lemmas~\ref{lem:splitplaces}a) and \ref{lem:infplaces}a).
If, additionally, $F_1/K$ and $F_2/K$ have identical local behavior at all other primes of multiplicative reduction of $E$ as well as at the archimedean primes, then we have $W(E/F_1) \ne W(E/F_2)$.

Assume now instead that there exists a prime $p$ at which $E$ has non-split multiplicative reduction. 
Then, from Lemma~\ref{lem:splitplaces}b), $W_p(E/K)$ can be controlled as soon as there exists pairs of subgroups $(I_1,H_1)$ and $(I_2, H_2)$ occurring as inertia group and decomposition group at $p$ in suitable $G$-extensions $F_1/K$ and $F_2/K$, and such that the parities of $|orb(H_1)| - |orb(I_1)|$ and $|orb(H_2)| - |orb(I_2)|$ are different. 
If, additionally, $I_2$ and $H_1$ are both cyclic and contained in $A_n$ (i.e., $n-|orb(I_2)|$, $n-|orb(H_1)|$ and $n-|orb(I_1)|$ are all even), 
then we obtain exactly the same sufficient condition ``$n-|orb(H_2)|$ odd" for the inequality of local root numbers $W_p(E/F_1) \ne W_p(E/F_2)$ as above. 

In total, we have shown the following simplified assertion:

\begin{prop}
\label{prop:basic}
Let $E$ be an elliptic curve over a number field $K$ 
and $\mathcal{S}_E$ be any finite set of primes of $K$, containing the archimedean places of $K$ and the primes of bad reduction of $E$. 
%
%
Let $G\le S_n$ be a transitive permutation group.

Assume the following:
\begin{itemize}
\item[i)] There exist $p\in \mathcal{S}_E$, either real archimedean or of multiplicative reduction for $E$, and a subgroup $H\le G$ which occurs as the Galois group of some extension of $\widehat{K_p}$ such that $n-|orb(H)|$ is odd. 
\item[ii)] There exist $G$-extensions $F_1/K$ and $F_2/K$, with isomorphic completions $\widehat{(F_1)_q} \cong \widehat{(F_2)_q}$ for all $q\in \mathcal{S}_E\setminus\{p\}$, and such that $p$ is unramified in $F_1/K$ with decomposition group contained in $G\cap A_n$, whereas the decomposition group at $p$ of $F_2/K$ equals $H$, and the corresponding inertia group $I\trianglelefteq H$ is cyclic and contained in $G\cap A_n$.
\end{itemize}
Then $W(E/F_1)\ne W(E/ F_2)$. In particular, conditional on the parity conjecture, $E$ gains rank over at least one of $F_1$ and $F_2$.
%
\end{prop}

\begin{remark}
Some words on the group-theoretical condition on the existence of a subgroup $H$ as in condition i) are in order. Obviously, it can be fulfilled for any subgroup $G\le S_n$ not contained in $A_n$. On the other hand, it cannot possibly be fulfilled for groups $G$ of odd order. An ``obvious" class of interesting test cases (i.e., for which the group theoretical condition is neither trivial nor impossible) is then the class of non-abelian simple groups, and we will follow this consideration in Sections \ref{sec:psl} and \ref{sec:an}. 
The group-theoretical condition still does not hold for all simple groups - e.g., computation with Magma \cite{Magma} confirms that $PSL_2(11)$ acting primitively on $55$ points has no metacyclic subgroup (i.e., no subgroup which could occur as the decomposition group at a tamely ramified prime) with $55-|orb(H)|$ odd. At least, it seems very frequently fulfilled among primitive simple groups.
\end{remark}

\subsection{Proof of Theorem \ref{thm:grunwald}}
%
%
Deducing positive answers to Question~\ref{ques:weak}a) (and, in fact, b)) for a group $G$ from Proposition \ref{prop:basic} is possible as soon as $G$ contains $H$ as in Proposition~\ref{prop:basic}, occurring as a Galois group over suitable completions $\widehat{K_p}$, and assuming the existence of $G$-extensions with suitable local behavior as in Condition ii) of Prop.\ \ref{prop:basic}. Tools to guarantee the latter are provided by Theorems \ref{thm:dg} and \ref{thm:spec_fctfd}, which we now use to prove Theorem \ref{thm:grunwald}. 

\begin{proof}[Proof of Theorem \ref{thm:grunwald}]
In case $G$ is not contained in $A_n$, define $I=\{1\}$ and $H$ to be any cyclic subgroup of $G$ not contained in $A_n$. If on the other hand $G\le A_n$, we are in Alternative ii) of the assumptions of Theorem \ref{thm:grunwald}, and set $H:=\langle I,x\rangle$ as defined there. Note that in both cases, $n-|orb(H)|$ is odd, and $I$ is cyclic and contained in $A_n$.

We now claim that for all primes $p$ of $K$ inside some (to be determined) set $\mathcal{S}$ of primes of positive density, and for all finite sets $\mathcal{S}_0$ of primes of $K$ different from $p$, there exist two specializations $L_{t_0(p)}/K$ and $L_{t_1(p)}/K$, both with Galois group $G$, such that 1) $L_{t_0(p)}/K$ has decomposition group $H$ and inertia group $I$ at $p$; 2)  $L_{t_1(p)}/K$ is unramified at $p$ with Frobenius contained in $G\cap A_n$; 
and 3) both extensions have identical local behaviour at each $q \in \mathcal{S}_0$.

Since the proof of the claim requires slightly different arguments in the two alternatives of Theorem \ref{thm:grunwald}, we assume first $G\le A_n$ (whence we are in Alternative ii)).
Then the first requirement of the claim is guaranteed by Theorem~\ref{thm:spec_fctfd}, assuming that we  choose $p$ satisfying both of the following:
\begin{itemize}
\item[a)] $p$ has Frobenius class (the conjugacy class of) $\overline{x}$ in the residue extension of $L/K(t)$ at the branch point $t\mapsto t_i$, where $\overline{x}$ denotes the image of the prescribed element $x\in D$ from Theorem~\ref{thm:grunwald} modulo $I$. 
\item[b)] $p$ is outside some finite set of primes depending on $L$.
\end{itemize}
Clearly the set $\tilde{\mathcal{S}}$ of such primes $p$ is of positive density by Chebotarev's density theorem, yielding the claim about $L_{t_0(p)}/K$.

Next, the second requirement of the claim simply translates to $p$ being unramified in $L_{t_1(p)}/K$, and hence can obviously be fulfilled for all but finitely many primes $p$ by choosing an arbitrary specialization value $t_1(p)\in K$. Define $\mathcal{S}$ as $\tilde{\mathcal{S}}$ minus these finitely many primes.

Now we need to find infinitely many values $t_0$ (resp., $t_1$) $\in K$ whose specialization $L_{t_0}/K$ (resp., $L_{t_1}/K$) has the same local behavior as $L_{t_0(p)}/K$ (resp., as $L_{t_1(p)}/K$) at $p$, retains the full Galois group $G$, and moreover such that $L_{t_0}/K$ and $L_{t_1}/K$ have identical local behavior at all $q\in \mathcal{S}_0$. 
This, however, is provided by Proposition~\ref{prop:hilbert_wa}, which also gives linear disjointness.
The assertion now follows immediately from Proposition \ref{prop:basic}. Concretely, choosing $F_1$ (resp., $F_2$) as the subfield of $L_{t_0}$ (resp., $L_{t_1}$) fixed by a point stabilizer in $G\le S_n$, one has $W(E/F_1)\ne W(E/F_2)$ for all elliptic curves $E$ over $K$ with at least one prime $p\in \mathcal{S}$ of multiplicative reduction (of course, we here apply the above with $\mathcal{S}_0$ the union of the set of archimedean primes of $K$ with the primes of bad reduction $q\ne p$ of $E$).

Now assume instead that we are in Alternative i) of the theorem, i.e., $G$ is not contained in $A_n$. 
Again, the assertion follows directly from Proposition \ref{prop:basic}, assuming that we can prove the claim. The only changes in the proof are required in the choice of specializations $L_{t_1(p)}/K$ and $L_{t_2(p)}/K$ fulfilling the first and second condition of the claim, respectively. More precisely, it now suffices to find specializations unramified at $p$ and with a prescribed conjugacy class as Frobenius at $p$ (namely, a class of odd and even permutations for conditions 1) and 2), respectively). This is possible due to Theorem \ref{thm:dg}, thus completing the proof.
\end{proof}

Note that in the case when $G$ is not contained in $A_n$, the above proof actually yields a cofinite set $\mathcal{S}$ of primes fulfilling the assertion of Theorem \ref{thm:grunwald}. 
We end this section by pointing out a class of groups for which one has an even stronger conclusion.

\begin{theorem}
\label{thm:generic}
Assume the parity conjecture.\\
Let $K$ be a number field, and let $G\le S_n$ be a transitive group not contained in $A_n$ and possessing a generic Galois extension over $K$. Let $E$ be an elliptic curve over $K$, and assume that 
$E$ possesses at least one prime of multiplicative reduction. 
Then $E$ gains rank over infinitely many $G$-extensions of $K$.
\end{theorem}

\begin{proof}
We may choose a prime $p$ of $K$, 
 of multiplicative reduction for $E$, and an element $\sigma\in G$ with $\textrm{sgn}(\sigma)=-1$. 
 Since $G$ is assumed to possess a generic Galois extension over $K$, we may invoke a famous result by Saltman \cite[Theorem 5.9]{Saltman}, asserting the solvability of all ``Grunwald problems" for $G$ over $K$. I.e., given  any finite set $S$ of places of $K$ and Galois extensions $\widehat{L^q}/\widehat{K_q}$ with Galois group embedding into $G$ (for all $q\in S$), there exists a $G$-extension of $K$ with completion $\widehat{L^q}/\widehat{K_q}$ for all $q\in S$ simultaneously. 
 It now suffices to choose all the $\widehat{L^q}/\widehat{K_q}$ unramified and simply apply Proposition \ref{prop:basic} with $H=\langle\sigma\rangle$ and $I=\{1\}$, yielding $G$-extensions $F_1/K$ and $F_2/K$ with $W(E/F_1)\ne W(E/F_2)$. 
 Of course, one even obtains infinitely many pairs of such extensions, e.g., up to increasing the finite set $S$ and thus the number of local conditions above suitably. This completes the proof. 
\end{proof}

\begin{remark}
\begin{itemize}
\item[a)]
Groups for which generic Galois extensions exist over all number fields $K$ include the  symmetric groups $S_n$, whence in particular we immediately regain a positive answer to Question \ref{ques:strong} over (e.g.) $K=\mathbb{Q}$, although of course this follows unconditionally from \cite{LOT}. Another class of examples are dihedral groups $D_p$ for all primes $p$, cf.\ \cite[Proposition 5.5.2]{JLY} (and these groups are not contained in $A_p$ as soon as $p\equiv 3$ mod $4$).
\item[b)] If $G$ is even a semidirect product of $G\cap A_n$ and $C_2$ (such as the two examples in a)), then even the condition on $E$ to have a prime of multiplicative reduction in Theorem \ref{thm:generic} can be dropped, as long as $K$ has a real embedding. Indeed, one may then choose a real-archimedean prime $p$ and a cyclic subgroup $H$ generated by an involution in $G\setminus A_n$ (taking the role of complex conjugation in an extension of $\widehat{K_p}$), and carry out the proof in analogy to the above. 
\end{itemize}
\end{remark}

In light of the above, it is more challenging to focus on groups $G\le A_n$, and for which generic Galois extensions are not known to exist (for example, the existence of generic $A_n$-extensions over $\mathbb{Q}$ is open for all $n\ge 6$, see (8.5.12) in \cite{JLY}).

\subsection{Application of Theorem \ref{thm:grunwald}: A first example}
\label{sec:firstex}

We show the practical applicability of Theorem \ref{thm:grunwald} with a first non-trivial example.  In the next section we will give an infinite family of examples.
 It should be understood that applications of Theorem \ref{thm:grunwald} are not limited in any way to these cases, and the interested reader may consult, e.g., \cite{MM} for many more $\mathbb{Q}$-regular Galois realizations of simple groups $G$, some of which (with suitable permutation representations of $G$) lend themselves to similar applications. Group-theoretical claims in the following example were verified with Magma.

\begin{example}
Let $G=M_{24} < A_{24}$ be the largest sporadic Mathieu group in its natural primitive action on $24$ points. We show the assumptions of Theorem \ref{thm:grunwald} can be fulfilled for $G$ over $K=\mathbb{Q}$. By \cite[Theorem III.7.12]{MM}, there exists a $\mathbb{Q}$-regular Galois realization $L/\mathbb{Q}(t)$ with group $G$ and with four branch points, of ramification index $2$, $2$, $2$ and $12$ respectively. The branch point $t_1$ of ramification index $12$ is automatically $\mathbb{Q}$-rational. Indeed, this follows from a special case of the so-called branch cycle lemma (e.g., \cite[Chapter I, Theorem 2.6]{MM}), asserting that branch points conjugate to each other must have the same ramification index. Furthermore, the inertia group $I$ of order $12$ has exactly two orbits (each of length $12$) in the action on $24$ points. The residue extension of $L/\mathbb{Q}(t)$ at $t\mapsto t_i$ must contain the $12$-th roots of unity (see, e.g., \cite[Lemma 2.3]{KLN19}), whence the decomposition group at $t\mapsto t_i$ is of order divisible by $12\cdot \varphi(12) = 48$. On the other hand, the normalizer of the cyclic subgroup $I$ in $M_{24}$ is only of order $48$, whence equality must hold. This normalizer contains a subgroup $H=\langle I, x\rangle$ isomorphic to the dihedral group $D_{12}$, such that $x$ switches the two orbits of $I$; in other words, $24-|orb(H)| = 23$ is odd. This shows that the assumptions of Theorem \ref{thm:grunwald} are fulfilled.
\end{example}

Note that the above argument used only ramification indices of a certain $\mathbb{Q}$-regular realization; also, the proof of the existence of this particular realization can be carried out purely group-theoretically (using braid group action). The argument therefore does not require the computation of a concrete polynomial for the extension $L/\mathbb{Q}(t)$.

\subsection{Projective linear groups: Proof of Theorem \ref{thm:psl}}
\label{sec:psl}

\begin{proof}

We can make Theorem \ref{thm:grunwald} more explicit using a concrete $PSL_2(p)$-extension $L/\mathbb{Q}(t)$ as given in \cite[Chapter 1, Cor.\ 8.10]{MM}. This extension has three branch points, 
of ramification index 
 $2$, $p$ and $p$ respectively. Since there is only one branch point $t_1$ of ramification index $2$, it is automatically $\mathbb{Q}$-rational by the branch cycle lemma \cite[Chapter I, Theorem 2.6]{MM}. 
 Furthermore, as shown in the proof of \cite[Theorem 5.3]{KLN19}, the residue extension $L_{t_1}/\mathbb{Q}$ at this branch point must contain $\sqrt{-1}$. The same then remains true for the $K$-regular $PSL_2(p)$-extension $L':=L\cdot K/K(t)$. In particular, for any prime $q$ of $K$ which is inert in $K(\sqrt{-1})$, the Frobenius of $L'_{t_1}/K$ at $q$ has even order. Therefore, we can use Theorem \ref{thm:spec_fctfd} to obtain the following:\\
For all but finitely many primes $q$ of $K$ inert in $K(\sqrt{-1})$, there exist infinitely many $t_0\in K$ such that $L'_{t_0}/K$ has inertia group $I$ of order $2$ and decomposition group $\langle I,x\rangle$ of order divisible by $4$ at $q$. However, due to our condition $p\equiv 3 \pmod 8$, the only class of elements of order $2$ in $PSL_2(p)$ is the class of fixed point free involutions, which are contained in a cyclic subgroup of order $\frac{p+1}{2}$. The centralizer of such an involution is the normalizer of such a cyclic subgroup, and is isomorphic to a dihedral group of order $p+1$ (acting regularly on $p+1$ points). Cf., e.g., \cite[Chapter II.8]{Huppert} (in particular Satz 8.4 and Satz 8.5), for the above well-known facts on subgroups of $PSL_2(p)$.
Now the group $H:=\langle I,x\rangle$ is a subgroup of such an involution  centralizer, and thus has $\frac{p+1}{|H|}$ orbits (each of length $|H|$). Since $p+1 \equiv 4 \pmod 8$ by assumption and $|H|$ is divisible by $4$ by the above, this means that the number of orbits of $H$ is odd. On the other hand, we are considering $PSL_2(p)\le S_{p+1}$ acting transitively on a set of even cardinality. Now, as in the proof of Theorem \ref{thm:grunwald} we obtain that for all but finitely many $q$ as above, and for all elliptic curves $E$ with multiplicative reduction at $p$, there exist infinitely many $PSL_2(p)$-extensions of $K$ over which $E$ gains rank. This concludes the proof.
\end{proof}

\begin{example}
\label{ex:psl}
Concrete polynomials for the $PSL_2(p)$-realizations used in the above proof are available in small cases. Notably, for $p=11$, \cite[Theorem 3]{Malle} provides the polynomial 
$$f(t,X) = 
(X^3-66X-308)^4-9s(t)\cdot (11X^5-44X^4-1573X^3+1892X^2+57358X+103763)$$$$-3s(t)^2(X-11)\in \mathbb{Q}(t)[X] \ 
\text{ (where } s(t):=2^8\cdot 3^5/(11t^2+1)),$$ with three branch points of ramification index $2$, $11$ and $11$. Indeed, one verifies that $t=\infty$ is the branch point of ramification index $2$. Let $E$ be an elliptic curve over $\mathbb{Q}$ having a ``sufficiently large" prime $q\equiv 3$ mod $4$ of multiplicative reduction. Following the proof of Theorem \ref{thm:grunwald}, to obtain different parities of the rank over (degree-$12$) $PSL_2(11)$-extensions $K_1/\mathbb{Q}$ and $K_2/\mathbb{Q}$, it suffices to choose specialization values $t_1,t_2\in\mathbb{Q}$ such that $t_1$ and $\infty$ meet at $q$ with multiplicity $1$ (i.e., $q$ strictly divides the denominator of $t_1$), $q$ is unramified in the specialization at $t_2$; and $t_1$ and $t_2$ are sufficiently close to each other modulo all further bad primes of $E$ (including $\infty$).
Following Remark \ref{rem:explicit_primes}, one may verify that the exceptional set $\mathcal{S}_0$ of primes in Theorem \ref{thm:spec_fctfd} may actually be taken as $\{2,3,5,7,11\}$ in this example. Furthermore all primes $q>11$ are unramified in {\it some} specialization (e.g., already the splitting field of $f(1,X)$ is unramified at all $q>11$). Thus, the condition on $q$ to be ``sufficiently large" may in fact be simplified to ``$q>11$".
 \end{example}

 \subsection{Alternating groups: Proof of Theorem \ref{thm:an_cond}}
\label{sec:an}

For the case of alternating groups, we can make use of strong results on regular extensions with prescribed fibers due to Mestre, and thus reach a stronger conclusion than in the previous cases.

The following special existence result on $A_n$-extensions with prescribed local conditions will suffice for our purposes.
\begin{lemma}
\label{lem:localglobal}
Let $K$ be a number field, $S = \{p_1,\cdots, p_r\}$ be a non-empty finite set of primes of $K$, such that $S\setminus\{p_1\}$ contains all archimedean primes of $K$. Let $n\ge 4$. 
Then there exist infinitely degree-$n$ extensions $F_1/K$ and $F_2/K$ with group $A_n$ fulfilling the following:
\begin{itemize}
\item[a)] The decomposition group at any prime extending $p_1$ in the Galois closure of $F_1/K$ is conjugate in $A_n$ to the Klein $4$-group on $\langle(1,2)(3,4), (1,3)(2,4)\rangle$, and the respective inertia subgroup is conjugate to $\langle (1,2)(3,4)\rangle$.
\item[b)] $p_i$ splits completely in $F_j/K$, for all $i=2,\dots, r$, $j\in \{1,2\}$ and $(i,j)\ne (1,1)$.
\end{itemize}
\end{lemma}

Theorem \ref{thm:an_cond} immediately follows from Lemma \ref{lem:localglobal},
together with Proposition~\ref{prop:basic}. It therefore remains to prove Lemma~\ref{lem:localglobal}.

\begin{proof}

First, let $n$ be odd. Clearly, there is a $V_4$-extension of $K$ (with $V_4$ the Klein $4$-group) with local behaviour in $S$ as prescribed for $F_1/K$ in a) and b), e.g., as a very special case of the Grunwald-Wang theorem (e.g., \cite[Theorem 9.2.8]{NSW}).
Denote this extension by $F/K$. Now a famous construction of Mestre \cite{Mes90} ensures the existence of an $A_n$-extension $L/K(t)$ which specializes to $F/K$ at $t\mapsto 0$ (and is unramified at $t\mapsto 0$). More precisely, $L/K(t)$ can be chosen as the splitting field of a polynomial $f(X)-tg(X)$ where $f$ is separable of degree $n$ with group $V_4$, and $\deg(g)<n$.
Now, specialize $L/K(t)$ at any value $t_0\in K\setminus\{0\}$ which is $p$-adically ``sufficiently" close to $0$ for all $p\in S$.
The assertion for the infinitely many fields $F_1/K$ then follows from compatibility of weak approximation with Hilbert's irreducibility theorem, as stated in Proposition \ref{prop:hilbert_wa}. 
To obtain the same for the fields $F_2/K$, simply repeat the above starting with the trivial extension $K/K$ instead of $F/K$.

Now let $n$ be even. Then choose an extension $L/K(t)$ as above for the group $A_{n+1}$, and note that by construction the fixed field of $L$ is a rational function field, say $K(X)$. Without loss, assume that $X\mapsto 0$ is a point extending $t\mapsto 0$ in $K(X)/K(t)$. Then specialization of $L/K(X)$ at $X\to 0$ yields a $V_4$-extension, and we may now specialize $X$ $p$-adically close to $0$ as above to reach the desired conclusion for $A_n$ with $n$ even, completing the proof.
\end{proof}

\begin{remark}
Note that while the extensions $L/K(t)$ constructed in the above proof depend on the concrete elliptic curve $E$, their genus does not, whence Remark \ref{rem:density} yields an asymptotic lower bound independent of $E$ for the number of $A_n$-extensions over which $E$ gains rank.
\end{remark}

\begin{example}
Given a concrete elliptic curve and a concrete $n\ge 5$, $A_n$ extensions as described above may be computed explicitly. To demonstrate this, let $E$ be the (rank $0$) elliptic curve over $\mathbb{Q}$ given by
$Y^2 = X(X-1)(X-5)$, and let $n=5$. This curve has good reduction outside $2$ and $5$, with local root numbers $W_2(E) = +1$ and $W_5(E) (=W_\infty(E)) = -1$. Following the proof of Lemma \ref{lem:localglobal}, we begin with a $V_4$-extension of $\mathbb{Q}$ in which the prime $5$ (of split multiplicative reduction type for $E$) has decomposition group $V_4$, and in which $2$ and $\infty$ split. For example, $\mathbb{Q}(\sqrt{65}, \sqrt{17})/\mathbb{Q}$ is such an extension, and one may verify easily that this is the splitting field of the (degree-$5$, increased artificially in order to embed into $A_5$) $V_4$-polynomial $f(X)=X(X^4-164X^2+2304)$. Mestre's construction guarantees the existence of a degree-$4$ polynomial $g(X)$ such that $f(X) - tg(X)$ has Galois group $A_5$ over $\mathbb{Q}(t)$, and one may compute explicitly that $g(X)=X^4-559076/12249X^2+2975047936/16670889$ is such a polynomial. (Concretely, $g(X)$ is found as a solution of $f'g - g'f = r^2$ for a suitable polynomial $r$, see \cite[Prop.\ 1]{Mes90}.) 
Therefore, a root field of $f(X)-t_0g(X)$ for any $0\ne t_0\in \mathbb{Q}$ sufficiently close to $0$ in the $2$-adic, $5$-adic and real absolute value (and with non-shrinking Galois group) may be used as a field $K_1$ as in Lemma \ref{lem:localglobal}. On the other hand, we may begin with the trivial extension (for example, as the splitting field of the degree-$5$ polynomial $X(X^2-1)(X^2-4)$). We then obtain the polynomial $X(X^2-1)(X^2-4) -t(X^4-73/29X^2+21316/21025)$ with Galois group $A_5$ over $\mathbb{Q}(t)$, and a root field of the specialization at any specialization value $t_0$ sufficiently close to $0$ in the $2$-adic, $5$-adic and real absolute value (and with non-shrinking Galois group) may be used as a field $K_2$ as in Lemma \ref{lem:localglobal}. By Proposition \ref{prop:basic}, the root numbers of $E$ over $K_1$ and over $K_2$ are different, whence $E$ gains rank over one of them. In fact, since the fields $K_1$ are chosen such two (resp., five) primes are extending the rational prime $5$ (resp., infinity), it follows  that 
$$W(E/K_1) = \underbrace{1^5}_{ \text{ contribution at } 2} \cdot \underbrace{(-1)^2}_{ \text{ contribution at } 5} \cdot \underbrace{(-1)^5}_{ \text{ contribution at } \infty} = -1,$$ so rank gain will occur over the fields $K_1$.

\end{example}

\subsection{About the elliptic curves satisfying the assumptions of the theorems}
\label{sec:density}

As shown in \cite[Theorem 8]{RvB}, 100 percent of elliptic curves over $\mathbb{Q}$ have at least one prime of multiplicative reduction, and thus fulfill the conclusion of Theorem \ref{thm:an_cond}. In fact, what is shown in \cite{RvB} is that for any finitely many prime numbers $p_1,\dots, p_r$, the proportion of elliptic curves which do not have multiplicative reduction at any $p_i$ is bounded from above by $\prod_{i=1}^r (1-1/p_i + o(1/p_i))$, see the proof of \cite[Corollary 28]{RvB}.

It then follows easily that the ``$100\%$ density" conclusion holds even for elliptic curves with at least one prime of multiplicative reduction inside a prescribed positive density subset of all prime numbers. In other words, Theorems \ref{thm:grunwald} and \ref{thm:psl} also apply to 100 percent of elliptic curves over $\mathbb{Q}$, ordered by height. It is reasonable to expect the same over arbitrary number fields, although a formal proof of such a statement would lead us too far away.

\section{A generalization to abelian varieties}
\label{sec:abvar}
Of course, it is natural to ask Question \ref{ques:weak} for abelian varieties of dimension $>1$ as well. Our methods generalize to this case as long as one has sufficiently explicit formulae linking the local root numbers of such varieties to such group-theoretical invariants as orbit numbers of decomposition groups etc.
We exhibit some cases in which this can be achieved.

\subsection{Semistable Jacobians}
For the following, see Section 5 of \cite{BKS}.

Let $K$ be a number field, $p$ be a non-archimedean prime of $K$ and $K_p$ the completion of $K$ at $p$. Denote the residue field of $K_p$ by $k$. Let $A:=J(C)$ be the Jacobian of a smooth geometrically irreducible proper curve over $K_p$, and assume that the reduction $\overline{C}$ of $C$ modulo $p$ is a stable curve over $k$, in the sense of \cite{DM}. Then, following \cite[Section 5]{BKS}, one has an explicit formula for the local root number $W_p(A)$ of $A$ at $p$. To state it, we need to recall some notation:

Let $\Sigma$ be the set of (nodal) singular points of $\overline{C}$ in $\overline{k}$, and let $\tilde{C}$ denote the normalization of $\overline{C}$. 
One has a birational morphism $\eta:\tilde{C}\to \overline{C}$, and due to the assumption of stable reduction, one has $|\eta^{-1}(\{z\})|=2$ for every $z\in \Sigma$. Now furthermore, for each $z\in \Sigma$, denote the $G_k$-orbit of $z$ by $[z]$, and the residue extension of $z$ over $k$ by $k(z)$. Denote the two points in $\eta^{-1}(\{z\})$ by $x$ and $y$, and define a sign $\tau({[z]})$ to be $+1$ or $-1$, depending on whether the power of Frobenius $\varphi^d\in G_{k(z)}$ (where $d:=[k(z):k]$) fixes or switches the two points $x$ and $y$ (clearly, this sign depends only on $[z]$ and not on $z$ itself). Lastly, let $\Sigma_k$ be the set of $G_k$-orbits on $\Sigma$, define $n_k = |\Sigma_k|$, and let $s_k$ be the number of irreducible $k$-components of $\tilde{C}$. 

\begin{prop}$($\cite[Proposition 5.4]{BKS}$)$
Under the above assumptions,
$$W_p(A) = (-1)^{n_k+s_k+1} \prod_{[z]\in \Sigma_k} \tau({[z]})  = (-1)^{s_k+1}\cdot \prod_{[z]\in \Sigma_k} (-\tau({[z]})).$$ 
\end{prop}

We can now state the following somewhat technical result, which may be seen as an analog of Proposition~\ref{prop:basic}.

\begin{lemma}
\label{lem:jac_techn}
Let $K$ be a number field and $G\le S_n$ be a transitive group, and assume that all of the following hold:
\begin{itemize}
\item[1)] There exist $G$-extensions $F_1/K$ and $F_2/K$ such that the parities of the number of primes extending $p$ in $F_1$ and in $F_2$ are different.
\item[2)] There exist $G$-extensions $F_1'/K$ and $F_2'/K$ such that no prime extending $p$ in $F_1'$ or in $F_2'$ has residue degree divisible by $4$, and such that the parities of the number of primes extending $p$ {\it of even residue degree} in $F_1'$ and in $F_2'$ are different.
\end{itemize}

Furthermore, let $C$ be a smooth geometrically irreducible curve over $K$ and $A=J(C)$ its Jacobian. 
For a prime $q$ of $K$ and a finite extension $L/K$, define $W_q(A\otimes L_i)$ to be the product of all local root numbers $W_{q'}(A\otimes L)$, where $q'$ runs through the primes extending $q$ in $L$.
\\
Assume there exists at least one non-archimedean prime $p$ of $K$ such that all of the following hold:
\begin{itemize}
\item[i)] The reduction $\overline{C}$ of $C$ at $p$ is a stable 
curve with an odd number of components, all absolutely irreducible.
\item[ii)] There exists exactly one $G_{k_p}$-orbit $[z]$ of singular points of $\overline{C}$.
\item[iii)] Either $\tau_p([z])=-1$, or $[z]$ is of odd residue degree over $k_p$.
%
\end{itemize}






Then there exist $G$-extensions $L_1/K$ and $L_2/K$ such that $W_p(A\otimes L_1)\ne W_p(A\otimes L_2)$.  
\end{lemma}
\begin{proof}
For a finite extension $L_1$ of $K$, let $\mathfrak{q}$ be a prime extending $p$ in $L_1$. Denote by $k_{\mathfrak{q}}$ and $k_p$ the residue fields at $\mathfrak{q}$ and $p$ respectively, and let $d=[k_{\mathfrak{q}}:k_p]$ be the residue degree. The stable model $\overline{C}$ is well-known to be invariant under base change, and so any $k_p$-orbit of singular points $[z]\in \Sigma_{k_p}$ of $\overline{C}$ of degree $d(z)$ decomposes into $\gcd(d,d(z))$ $k_{\mathfrak{q}}$-orbits $[z']$ of singular points, all of which are of degree $\frac{d(z)}{\gcd(d,d(z))}$.
Set $\omega({[z]}):= - \tau_{[z]}$, and let $\omega_{\mathfrak{q}}([z])$ be the product of all $\omega([z'])$, with $[z']$ a $k_{\mathfrak{q}}$-orbit dividing $[z]$ as above. Then \begin{equation}
    \label{eq:1}
\omega_{\mathfrak{q}}([z]) = (-1)^{\gcd(d,d(z))} \cdot \tau_{\mathfrak{q}}([z']),\end{equation}
where $[z']$ is any fixed $k_{\mathfrak{q}}$-orbit dividing $[z]$. 
Let $x$ be a point over $z$ in the cover $\tilde{C}\to \overline{C}$ as above, and denote the corresponding residue extension by $k_p(x)/k_p(z)$.
We have the following possibilities for the pair $(\tau([z]), \tau_{\mathfrak{q}}([z']))$:

\begin{equation}
\label{eq:2}
(\tau([z]), \tau_{\mathfrak{q}}([z'])) = \begin{cases}
(-1,1), \text{ if } [k_p(x):k_p(z)] =2 \text{ and } d \text{ is divisible by the maximal 2-power } \\ \hspace{30mm}\ 2^a \text{ dividing } 2d(z)=[k_p(x):k_p],\\
(-1,-1), \text{ if } [k_p(x):k_p(z)] =2 \text{ and } d \text{ is not divisible by } 2^a,\\
(1,1), \text{ if } k_p(x)=k_p(z).
\end{cases}\end{equation}
Including the factor $(-1)^{\gcd(d,d(z))}$ and going through all possible combinations of $\omega([z])$, $d$ and $d(z)$, we get from \eqref{eq:1} and \eqref{eq:2} that



\begin{itemize}
\item[a)] If $\omega([z]) = -1$, then $\omega_{\mathfrak{q}}([z]) =1$ if and only if $d$ and $d(z)$ are both even.
\item[b)] If $\omega([z]) = +1$, then $\omega_{\mathfrak{q}}([z]) =-1$ if and only if $d$ is even, and $d(z)$ either odd or divisible by $2$ at least as often as $d$ is.
\end{itemize}

Assume now additionally that $\overline{C}$ has an odd number of components, all absolutely irreducible, i.e., $(-1)^{s_{k_p}} = (-1)^{s_{k_\mathfrak{q}}}=-1$; and that there exists only a single $k_p$-orbit $[z]$ of singular points. Then in case a) above, $W_p(A) = -1$. 
If additionally the degree $d(z)$ is odd (which we may assume for case a) due to condition iii)) , then simply $\prod_{\mathfrak{q}|p} W_{\mathfrak{q}} (A\otimes L_i) = (-1)^{m}$ where $m$ is the total number of primes extending $p$ in $L_i$.

To deal with case b) above more efficiently, we make the additional assumption none of the primes extending $p$ in $L_i$ are of residue degree divisible by $4$ (in which case the condition on $d(z)$ in Case b) is automatically fulfilled). One then has $W_p(A) = 1$; and $\prod_{\mathfrak{q}|p} W_{\mathfrak{q}} (A\otimes L_i) = (-1)^{m_1}$, where $m_1$ is the number of primes extending $p$ of even residue degree.

In particular, if for any given one of the above values $j\in \{m, m_1\}$, we manage to choose $j$ of different parity in $L_1$ and in $L_2$, then $\prod_{\mathfrak{q}|p} W_{\mathfrak{q}} (A\otimes L_1) \ne \prod_{\mathfrak{q}|p} W_{\mathfrak{q}} (A\otimes L_2)$. Due to Assumptions 1) and 2), this can be achieved in either case. This concludes the proof.
%
%
%
%
%
%
%
\end{proof}


\begin{theorem}
\label{thm:semist_jac}
Assume the parity conjecture for abelian varieties. Let $C/K$ be a curve over a number field $K$ fulfilling the assumptions i), ii) and iii) of Lemma~\ref{lem:jac_techn}. Then 
\begin{itemize}
    \item[a)] For each  integer $n\ge 4$, there exist infinitely many $A_n$-extensions over which $J(C)$ gains rank.
    \item[b)] If additionally $K$ does not contain $\sqrt{-1}$, then for any prime $\ell \equiv 3 \pmod 8$, the following holds: if the prime $p$ in Lemma \ref{lem:jac_techn} is of sufficiently large norm (depending on $\ell$) and inert in $K(\sqrt{-1})$, then there exist infinitely many (degree $\ell+1$) $PSL_2(\ell)$-extensions of $K$ over which $J(C)$ gains rank.
    \end{itemize}
    \end{theorem}
\begin{proof}
It suffices to verify that the extensions constructed in the proofs of Theorem \ref{thm:an_cond} and Theorem \ref{thm:psl_strong}, are such that they fulfill conditions 1) and 2) of Lemma \ref{lem:jac_techn} for the prescribed prime $p$ (and for $G=A_n$ and $G=PSL_2(\ell)$ respectively). Indeed, the proofs yielded (via using Proposition \ref{prop:hilbert_wa}) that additionally the local behaviors at a prescribed finite set of primes outside of $p$ could be demanded to remain unchanged, whence the root number $W(J(C)\otimes L_i)$ is controlled entirely by the local root number $W_p(J(C)\otimes L_i)$; the assertion then follows from Lemma \ref{lem:jac_techn}.

For $G=A_n$, recall that we could choose $L_1/K$ to have decomposition group $D=\langle(1,2)(3,4)$, $(1,3)(2,4)\rangle\le A_n$ and inertia group $I=\langle (1,2)(3,4)\rangle$ at $p$, and $L_2/K$ such that $p$ splits completely in $L_2$. Then $L_1/K$ has exactly one prime of residue degree $2$ extending $p$ (namely corresponding to the orbit $\{1,2,3,4\}$ of $D$), whereas $L_2$ has none. All other primes over $p$ are of residue degree $1$, and their number is $n-4$ in $L_1$ and $n-2$ in $L_2$. This shows that assumptions 1) and 2) of Lemma~\ref{lem:jac_techn} are fulfilled, whence $W_p(J(C)\otimes L_1) \ne W_p(J(C)\otimes L_2)$. 

For $G=PSL_2(\ell)$, recall that we could choose $L_1/K$ such that an odd number of primes extend $p$ in $L_1$, all having residue degree $2$ (since the Frobenius of the maximal unramified subextension was chosen to be an (automatically fixed point free) involution), and again $L_2/K$ such that $p$ splits completely (into an even number of primes, namely $\ell+1$). This again shows that the assumptions of Lemma~\ref{lem:jac_techn} are fulfilled, concluding the proof.
%
%
\end{proof}

\subsection{Hyperelliptic Jacobians: proof of Theorem \ref{thm:hyper}}
We deduce Theorem \ref{thm:hyper} from Theorem \ref{thm:semist_jac}.
\begin{proof}
Let $k=\mathbb{F}_p$. Under our assumptions, the reduction of $C$ at $p$ is birational to $Y^2=g(X)$, and in particular absolutely irreducible. We thus have  $s_{\tilde{k}}=1$ for any finite extension $\tilde{k}$ of $k$. Also $n_k=1 = n_{\tilde{k}}$, since there is only one singularity (namely the ordinary double point at $X=a$). 
Thus all assumptions i), ii) and iii) of Lemma~\ref{lem:jac_techn} are fulfilled, whence the assertion follows from Theorem~\ref{thm:semist_jac}.
\end{proof}

Note that the assumption in Theorem \ref{thm:hyper} on $f$ to have only one double root modulo $p$ for {\it some} $p$ is not strong in the sense that it is fulfilled as soon as the discriminant of $f$ has at least one simple prime divisor, which in term is fulfilled by 100 percent of hyperelliptic curves in a well-defined density sense, see \cite{RvB}.

\end{document}